\documentclass[12pt, letterpaper]{article}

\usepackage{fullpage}
\usepackage{amsfonts}
\usepackage{amsmath, amsthm}
\usepackage{amsthm}
\usepackage{amssymb}
\usepackage{cite}
 \usepackage{graphicx}
 \usepackage[all]{xy}

\newcommand{\F}{\mathbb{F}}
\newcommand{\Z}{\mathbb{Z}}
\newcommand{\C}{\mathbb{C}}
\newcommand{\Q}{\mathbb{Q}}
\newcommand{\R}{\mathbb{R}}
\newcommand{\A}{\mathbb{A}}

\newcommand{\p}{\mathfrak{p}}
\newcommand{\q}{\mathfrak{q}}
\newcommand{\Omf}{\mathfrak{O}}

\newcommand{\Aut}{\operatorname{Aut}}
\newcommand{\GL}{\operatorname{GL}}
\newcommand{\SL}{\operatorname{SL}}
\newcommand{\PGL}{\operatorname{PGL}}
\newcommand{\Gal}{\operatorname{Gal}}
\newcommand{\Eva}{\operatorname{G}}
\newcommand{\Imm}{\operatorname{Im}}
\newcommand{\Frob}{\operatorname{Frob}}
\newcommand{\Cl}{\operatorname{Cl}}
\newcommand{\disc}{\operatorname{disc}}
\newcommand{\Hom}{\operatorname{Hom}}

\usepackage{amsmath}
\usepackage{amssymb}
\usepackage{cite}

\newtheorem{theorem}{Theorem}
\newtheorem{lemma}{Lemma}
\newtheorem{proposition}{Proposition}

\newtheorem{remark}{Remark}

\newtheorem{acknowledgements}{Acknowledgements}

\begin{document}

\title{Solvable Artin representations ramified at one prime} 

\author{Jonah Leshin}

\maketitle

\begin{abstract}
 We classify the possibilities
  for the fixed field of the kernel of an irreducible three-dimensional Artin representation of $\Q$ with solvable image ramified at one
  prime by using the classification of the finite irreducible
  subgroups of $\PGL_3(\C)$. This allows us to bound the number of such representations
  with given Artin conductor.  

 \end{abstract}


\section{Introduction} 
\label{intro}
The Langlands program gives a correspondence between Galois
representations and automorphic representations. In the case of two-dimensional Artin representations,
the corresponding analytic objects
are modular forms of weight one and Maass forms of weight zero, which
correspond to odd and even two-dimensional Artin representations,
respectively. 
 Specifically, Deligne and Serre \cite{poids1} showed that given a
cuspidal modular form of weight $1$, level $N$, and Nebentypus
character $\epsilon$ with normalized $(a_1=1)$ Fourier expansion
$\sum a_nq^n$, there exists an Artin representation 
\[
\rho:\Eva_{\Q}:=\textrm{Gal}(\bar{\Q}/\Q)\to \GL _2(\C)
\]
of conductor $N$ such that (via $\rho$) the trace of the
conjugacy class of the Frobenius at $p$, $\Frob_p$, in
$\Eva_{\Q}$ is $a_p$, and the determinant of $\Frob_p$ is $\epsilon(p)$. 
Conversely, given a two-dimensional Artin representation $\rho$, if we insist that either
$\rho$ has solvable image \cite{fltbook} or that $\rho$ is odd \cite{khare}, then there is known to exist an automorphic form $f$ such that $f$ gives rise to $\rho$ in
the manner described above.  Using this correspondence, Wong \cite{siman}, building on
work of Duke \cite{duke}, bounded the number of two-dimensional Artin
representations of $\Q$ with Artin conductor $N$ in terms of $N$. 

\par Less is known in the case of three-dimensional Artin
representations. In this paper we bound the number of
three-dimensional irreducible Artin representations with solvable
image and prime power conductor. It is known that such representations come from automorphic
representations of $\GL_3(\A_{\Q})$, where $\A_{\Q}$ denotes
the ad\`eles of $\Q$. We use purely algebraic
techniques for our bounds, which therefore, in principle, could give bounds for the number of 
automorphic representations of $\GL_3(\A_{\Q})$ with
particular properties.
\par The main results of this paper, which are made more precise in
Theorems \ref{t1}, \ref{t2}, \ref{t4}, and \ref{t5} are:
\newline
\begin{itemize}
\item The number of irreducible, imprimitive three-dimensional Artin representations of
$\Eva_{\Q}$ with
conductor dividing $p^m$ is at most $Cmp^{m+1}$, where $C$ is an
explicitly given constant.

\item The number of solvable, irreducible, primitive three-dimensional Artin representations
  of $\Eva_{\Q}$ with conductor dividing $p^m$ is at most
  $Dp^{\frac{m}{3}+50}$, where $D$ is an explicitly given constant.\newline
\end{itemize}

\par The paper is divided into two main parts. The
first part looks at imprimitive representations. We use the fact
that any imprimitive three-dimensional representation is monomial to
bound the number of all imprimitive representations in one fell swoop.
The second part deals with primitive representations, where we do a
case by case analysis of the three groups that can occur as the
projective image of a primitive three-dimensional Artin
representation.

\section{Definitions and preliminaries}
Let $V$ be a finite dimensional vector space over an arbitrary field,
and let $G$ be a group. A representation $\rho:G\to \GL(V)$ is called
\textit{imprimitive} if there exist proper subspaces $V_1, \ldots, V_m$ of $V$
such that $V=\oplus_{i=1}^mV_i$ and $\rho(G)$ permutes the $V_i$. If
such a decomposition of $V$ does not exist, $\rho$ is called
\textit{primitive}. If $\rho$ is imprimitive, and if the corresponding $V_i$
are one-dimensional and permuted transitively by $G$, then $\rho$ is said to be \textit{monomial} (this
is equivalent to the standard definition of monomial, which is that
there exists a subgroup $H\leq G$ and a linear character $\chi$ of $H$
such that $\rho$ is induced from $\chi$). Thus monomial
representations of dimension greater than 1 are imprimitive, and the converse is true if the
dimension of the representation is prime, and if $G$ acts transitively
on the $V_i$. A group is said to be
\textit{monomial} if each of its irreducible representations is
monomial. A comprehensive treatment of these definitions and the general
theory of representations of finite groups can be found in \cite{CR}. 

An Artin representation of a number field $F$ is a continuous
homomorphism 
\[
\rho:\textrm{Gal}(\bar{F}/F)\to \GL_n(\C).
\]
We take the analytic
topology on $\GL_n(\C)$, so Artin representations
necessarily have finite image. The conductor $N(\rho)$ of $\rho$ is
defined to be the product of its local conductors, whose definition we
now recall. Let $\p$ be a prime
of $F$. For the $\p$-part $N_{\p}(\rho)$ of
$N(\rho)$, consider the restriction $\rho_{\p}$ of $\rho$ to
$\Eva(\bar{F}_{\p}/F_{\p})$ (if $E/F$ is an extension of fields, we will write $\Eva(E/F)$ in place of $\Gal(E/F)$ to ease notation). The representation $\rho_{\p}$ factors through a finite
quotient $\Eva(K/F_{\p})$ of $\Eva(\bar{F}_{\p}/F_{\p})$. Let $g_i$ denote the order of the $i$th
ramification group $G_i$ of $\Eva(K/F_{\p})$. We then define
$N_{\p}(\rho)$ to be $\p$ raised to the power

\[
\sum_{i\geq 0}\frac{g_i}{g_0}\textrm{dim}V/V^{G_i},
\]
where $V$ is the representation space for $\rho$ and $V^{G_i}$ is the
subspace of elements fixed by $G_i$. If $\mathfrak{a}$ is an ideal in
the ring of integers $\mathfrak{O}$ of a number field, such as the
conductor of an Artin representation, and if $\p$ is a
prime of $\mathfrak{O}$, we take $v_{\p}(\mathfrak{a})$ to be the
exponent of the highest power of $\p$ dividing $\mathfrak{a}$.  
 For the remainder of this paper we make the following hypothesis on
 all Artin representations:
\[
  \quad
  F=\mathbb{Q} \quad\textrm{and}\quad  n=3.
\]  Representations will always be distinguished up to
isomorphism. The natural projection $\pi:\GL_3(\C) \to \PGL_3(\C)$ will always be denoted
by $\pi$, and a $\textit{lift}$ of a projective
representation $\psi: G \to \PGL_3(\C)$ of a group $G$ is any
representation $\nobreak{\rho: G\to \GL_3(\C)}$ such that $\psi=\pi \circ
\rho$. A cyclic group of order $n$ will be denoted by $C_n$, and
$Z(G)$ denotes the center of a group $G$. If $L/K$ is a Galois
extension of fields, we say $E$ is a $\textit{central extension}$ of $L/K$ if
$E/K$ is Galois and $\Eva(E/L) \leq Z\big(\Eva(E/K)\big)$.
\par The symbol $E_{\rho}$ will denote the fixed field of the
kernel of one of our irreducible three-dimensional Artin
representations $\rho$ with solvable image,
and $F_{\rho}$ will denote the fixed field of the kernel of $\pi \circ
\rho$ (when the context is clear, the subscript $\rho$ may be
dropped). Thus $\rho$ descends to an injection $\Eva(E_{\rho}/\Q)
\hookrightarrow \GL_3(\C)$, which by abuse of notation we also denote
by $\rho$, and $Z\big(\Eva(E/\Q)\big)=\Eva(E/F)$, which is cyclic
since $\rho$ is irreducible (by Schur's Lemma \cite[Section 2.2]{linreps}).
\par We say an Artin representation $\rho$ is $\textit{ramified at }$$ p$ if the extension
$E_{\rho}/\Q$ is ramified at $p$. In sections \ref{Imprim} and \ref{Prim}, all
representations are assumed to be unramified outside $p\infty$,
where $p$ is a finite prime and $\infty$ is the real place of
$\Q$. To avoid special cases, most of which are trivial, we will assume $p\neq 2,3$. \par In our computations, we will wish to bound the number of $C_a$ or $C_a
\times C_a$ extensions of a number field $K$ with given ramification
conditions, where $a$ is prime. We do this by first finding the
$a$-rank of the maximal extension $L$ of $K$ with the given ramification
conditions. Let $r_{a,b}(x)$ be the number of $(C_a)^b$ extensions of
$K$ contained in a Galois extension $L$, where $\Eva(L/K)$ is abelian
and has $a$-rank
$x$. We will consider the cases when $a=2$ or $3$, although the
following formulas hold for any prime $a$: 
\[
r_{a,1}(x)=\frac{a^x-1}{a-1} \textrm{, \quad} r_{a,2}(x)=\frac{(a^x-1)(a^x-a)}{(a^2-1)(a^2-a)}.
\]
When we have an extension $L/K$ of number fields with $L/\Q$
and $K/\Q$ Galois, and $L/\Q$ ramified only at $p$, we will write
$e(L/K)$ to denote the ramification from $K$ to $L$ of the primes above $p$ in $K$. We will say an extension of fields is unramified if it is unramified at all finite primes. The ring of
integers of a number field $K$ will be denoted by $\Omf_K$ and its
absolute discriminant by $d_K$. We let $h_K$ denote the class
number of $K$, $\Cl(K)$ the ideal class group of $K$,
$\Cl^{\mathfrak{m}}(K)$ the ray class group of $K$ of modulus
$\mathfrak{m}$, $H_K$ the Hilbert class field of
$K$, and $K^{\mathfrak{m}}$ the ray class field of $K$ of
modulus $\mathfrak{m}$. If $\mathfrak{m}$ is a product of finite
primes of $K$, we will write $\mathfrak{m}\infty$ to denote the modulus that
is the product of $\mathfrak{m}$ multiplied by all the real places
of $K$.

\section{Imprimitive Representations}\label{Imprim}
In this section we assume that $\rho:\Eva_{\Q}\to \GL_3(\C)$ is
irreducible, imprimitive, and unramified outside $p\infty$. Since $\rho$ is monomial, there exists a field $L$ of
degree 3 over $\Q$ and a character $\chi:\Eva_{L} \to \C^*$ such that
$\rho=\chi_L^{\Q}$, where $\chi_L^{\Q}$ denotes the induced
representation of $\chi$ to $\Eva_{\Q}$. 
\par Suppose first that $L \subset
\Q(\zeta_p)$ and let $\p$ be the prime of $L$ above $p$. It follows
from the Corollary to Proposition 4, Chapter VI of \cite{MR554237} that
\begin{equation}
v_p(N(\rho))=v_{\p}(N(\chi))+2. \label{vp}
\end{equation}
 Let
\[
c_m= \# \{\textrm{Irreducible }\rho:\Eva_{\Q} \to \GL_3(\C)
\textrm{ such that } N(\rho)\mid p^m  \textrm{ and there exists }\chi:
  \Eva_{L} \to \C^* \]
\[\textrm{such
  that } \rho=\chi_L^{\Q}\},\]
\[
d_m=\#\{\chi:\Eva_L\to \C^* \textrm{ such that } N(\chi)\mid \p^m \}.
\]
It follows from Mackey's irreducibility criterion \cite[Section 7.4]{linreps} that the induction $\rho$ to $\Eva_{\Q}$ of a non-trivial $\chi \in \Hom(\Eva_L,
\C^*)$ will be irreducible exactly when $\rho$ restricted to $\Eva_L$
is the sum of three distinct characters, each of which induces to give
$\rho$ by Frobenius reciprocity \cite[Section 7.2]{linreps}. Therefore, using \eqref{vp}, we have 
\[ c_m\leq\frac{d_{m-2}}
{3}.\]
\par We first bound $d_m$. By
class field theory, any $\chi$ with $N(\chi)\mid \p^m$ factors through
$\Eva(L^{\p^m}/L)$, where $L^{\p^m}$ is the ray class field of $L$ of modulus
$\p^m$. Class field theory gives the exact sequence
\begin{equation}
 \Omf_L^* \to (\mathfrak{O}_L /\mathfrak{p}^{m}\infty)^*\to
 \Cl^{\mathfrak{p}^{m}\infty}(L) \to \Cl(L)
  \to 1,  \label{exact}
\end{equation} 
where $(\Omf_L/\p^{m}\infty)^*$ is defined to be $(\Omf_L /\p^m)^* \times
\langle \pm
1\rangle ^3$, the exponent $3$ coming from the 3 real places of
$L$. The image of $\Omf_L^*$ in $(\Omf_L/\p^{m}\infty)^*$ has order at least two, so we obtain
\[
d_m=|\Eva(L^{\p^m \infty}/L)|\leq 4h_Lp^{m-1}(p-1)< 4\left(\frac{1}{2}p\right)p^{m-1}(p-1)=2p^{m}(p-1),
\]
where the first inequality follows from \eqref{exact}, and the second
inequality from \cite[Corollary 4]{realcubic}, for which we need the fact that
$L$ is totally real. Thus we obtain 
\begin{theorem} \label{t1}
Notation as above,
$$ c_m< \frac{2p^{m-1}}{3}.$$
\end{theorem} 
\par For the remainder of this section, suppose that $L \nsubseteq \Q(\zeta_p)$, or equivalently,
that $L/\Q$ is not Galois. By the following
lemma, $p$ must split in $L$ as $\p_1^2\p_2$.
\begin{lemma}\label{lem3}
Let $L/\Q$ be a degree 3 extension ramified at a single finite prime $p\neq
2,3$. Then $L/\Q$ is Galois if and only if $p$ is totally ramified in $L$. 
\begin{proof}
The only if part is clear. For the converse, suppose $L/\Q$ is totally
ramified at $p$ and let $\p$ be the prime of $L$ above $p$. Let $L_{\p}$ be the completion of $L$ at $\p$ and
$\Omf=\Omf_{L_{\p}}$ the ring of integers in $L_{\p}$. Let $f(x)$ be the
minimal polynomial for $\alpha \in \Omf$, where $\alpha$
generates $\Omf$ as a $\Z_p$-algebra. We may assume $\alpha
\in \Omf_L$. The $p$-part of the
discriminant $\disc(f)$ of $f$ equals the local discriminant
$d_{L_{\p}/\Q_p}$, which has $p$-part
$p^2$ ($p \neq 3$), the same as the $p$-part of the global discriminant
$d_{L/\Q}$. So we can write $\disc(f)=\epsilon p^2D$, with $\epsilon
\in \{\pm 1\}$ and $p \nmid D$. The Galois closure $K$ of $L/\Q$ is
ramified only at $p$ and is given by
$K=L\big(\sqrt{\disc (f)}\big)=L\big(\sqrt{\epsilon D} \big)$. For $p\neq 2$, this is
only possible if $\sqrt{\epsilon D} \in L$, i.e., $L/\Q$ is Galois. 
\end{proof}
\end{lemma}
By the Corollary to Proposition 4 of \cite{MR554237}, we have
\[
N(\chi_L^{\Q})=N_{L/\Q}(N(\chi))d_{L}.
\]
It follows that if $N(\chi)=\p_1^{a_1}\p_2^{a_2}$, then
$N(\rho)=p^{a_1+a_2}p$. Let $c_m=c_{m,L}$ be the number of imprimitive
$\rho$ with conductor dividing $p^m$ that are induced from $\Eva_L$. To estimate $c_m$, we need to estimate
\[
\#\{\chi:N(\chi)\mid \p_1^{a_1}\p_2^{a_2}, a_1+a_2 \leq m-1\}.
\]
Thus we need to bound
$\sum_{a_1+a_2=m-1}|\Eva(L^{\p_1^{a_1}\p_2^{a_2}\infty}/L)|$. We have
\begin{equation}\label{mp}
c_m \leq \sum_{a_1+a_2=m-1}|\Eva(L^{\p_1^{a_1}\p_2^{a_2}\infty}/L)|\leq 4h_L\sum_{a_1+a_2=m-1}(p^{a_1}-1)(p^{a_2}-1)<4h_Lmp^{m-1},
\end{equation}
where the first inequality follows from the exact sequence
\eqref{exact} applied to the situation at hand. 
\par Since $L/\Q$
is not Galois, $\Eva_L$ is not normal in $\Eva_{\Q}$, so it is
possible that there may only be one $\chi \in \Hom(\Eva_L, \C^*)$ that
gives $\rho$ when induced up to $\Eva_{\Q}$. Since we are assuming that
$H:=\Eva(E_{\rho}/L)$ is not normal in $G:=\Eva(E_{\rho}/\Q)$, the
group $H$ has three conjugate subgroups inside $G$. Given a coset $sH$ and
$\chi \in \Hom(H,\C^*)$, we can define the representation $\chi^s \in
\Hom(sHs^{-1},\C^*)$ by $\chi^s(x)=\chi(s^{-1}xs)$. The formula for an
induced character \cite[Section 7.2]{linreps} shows that $\chi^s$
induced up to $G$ is isomorphic to $\chi$ induced up to $G$. The three
conjugate subgroups $H^s$ correspond to conjugates of $L$,
each degree three over $\Q$. Thus, it remains to bound $h_L$ and $\frac{n_p}{3}$, where $n_p$ is the
number of non-Galois fields $L$ of degree 3
over $\Q$, ramified only at $p\infty$. We do the latter first. Let $L$ be such a number field. The
Galois closure $K$ of $L/\Q$ is an $S_3$ extension of $\Q$ ramified only
at $p\infty$. By Lemma \ref{lem3}, $p$ splits as $\p_1^2\p_2$ in $L$. Let
$\q$ be a prime of $K$ above $\p_2$. We have 
\begin{equation}
e(\q/p)=e(\q/\p_2)e(\p_2/p)=e(\q/\p_2)\leq 2. \label{(*)}
\end{equation}
For any primes $\q_i, \q_j$ of $K$ above $p$, $e(\q_i/p)=e(\q_j/p)$,
and the quantity is at least 2, so by \eqref{(*)}, it equals 2. Thus
$K/\Q(\sqrt{p^{*}})$ is unramified, where $\Q(\sqrt{p^{*}})$ is the
subfield of $K$ of degree 2 over $\Q$. Therefore, $L$ and its two
Galois conjugates contained in $K$ combine to give one unramified $C_3$
extension of $\Q(\sqrt{p^{*}})$.
By \cite[page 95]{moonmonomial}, we have
$h_{\Q(\sqrt{p^*})}< 
\frac{22.2p}{\pi^2}$. Ramification at the infinite primes of
$\Q(\sqrt{p^*})$ can only contribute a power of 2 to the degree of the
maximal abelian unramified (at all finite primes) extension of
$\Q(\sqrt{p^*})$. Since we are interested in the $3$-rank of this
extension, we may ignore the potential ramification at the infinite primes. We have the obvious bound that the $3$-rank of
$\Cl\big(\Q(\sqrt{p^*})\big)$ is less than
\[
 \log_3\left({\frac{22.2p}{\pi^2}}\right),
\]
from which it follows that
\[
\frac{n_p}{3} <  r_{3,1}\left(\log_3\left({\frac{22.2p}{\pi^2}}\right)\right),
\]
which is less than
\[
\frac{11.1p}{\pi^2}. 
\]

\par Putting this last bound together with \eqref{mp}, and using the bound
$h_L < \frac{22.2p}{\pi^3}$ (which follows from \cite[page 95]{moonmonomial}),
we obtain
\begin{theorem}\label{t2} The number
of imprimitive $\rho$ with $\rho=\chi_L^{\Q}$, $N(\rho)$ dividing $p^m$, and $L/\Q$ non-Galois is
less than
$$
\left(\frac{985.7mp^{m+1}}{\pi^5}\right). 
$$
\end{theorem}
\begin{remark}
\emph{
It is possible to obtain a slightly better bound by using a
  slightly better bound for the $3$-rank of
  $\Cl\big(\Q(\sqrt{p^*})\big)$. See, for example, \cite{pierce}.}

\end{remark}

\section{Primitive representations} \label{Prim}
For this section, suppose that $\rho:\Eva_{\Q}\to \GL_3(\C)$ is
primitive. We continue with our assumption that $\rho$ is unramified
outside $p\infty$ and has solvable image. Since $\rho$ is primitive
with solvable image, the image of $\pi \circ \rho$ in $\PGL_3(\C)$
is isomorphic to one of the following groups \cite[Theorem 4.8]{pgl3}:
\[
P_1:=\big(C_3 \times C_3 \big)\rtimes C_4, \quad P_2:= \big(C_3
\times C_3 \big) \rtimes Q_8, \quad
 P_3:=\big( C_3 \times C_3 \big) \rtimes \SL_2(\F_3),
\]
where $Q_8$ is the quaternion group of order 8. We can view $P_1$ as a
subgroup of $P_2$, and $P_2$ as a subgroup of $P_3$.

\par We claim that $\Imm(\pi \circ \rho)$ cannot be $P_2$. 
If $\Eva(F_{\rho}/\Q)=P_2$, then $F_{\rho}$ (and thus
$E_{\rho}$) has a subfield $L$ with $L/\Q$ Galois and $\Eva(L/\Q)\cong
C_2 \times C_2$. But there is no such field ramified only at one prime
($p\neq 2$) of $\Q$.
\par In the next two subsections, we consider the cases $\Eva(F_{\rho}/\Q)=P_1$ and $\Eva(F_{\rho}/\Q)=P_3$. For each case, we first
fix a number field $F$ with $\Eva(F/\Q)=P_i$ and bound the number of
representations $\rho$ with $F_{\rho}=F$. We then bound the number of such
possible $F$-- i.e., $F$ such that $\Eva(F/\Q)=P_i$ and $F/\Q$ is
unramified outside $p\infty$.  
\par We will use the fact that the Schur multiplier $H_2(G,\Z)=C_3$
for $G=P_1$ and $G=P_3$. An alternative definition of the Schur
multiplier of a group $G$ is the largest group $X$ such that there
exists a group $\Gamma$ with $\Gamma/X \cong G$ and $X \subseteq \Gamma' \cap Z(\Gamma)$ \cite[page 151]{iscsgroup}, where $\Gamma'$ is the commutator
subgroup of $\Gamma$. Using this definition, one shows that the $p$-Sylow
subgroup of $H_2(G, \Z)$ is isomorphic a subgroup of the Schur
multiplier of a $p$-Sylow subgroup of $G$. Using this fact, it is an
exercise to show that $H_2(P_1, \Z)=C_3$. To show that
$H_2(P_3,\Z)=C_3$, we need to show that the Sylow-2 subgroup of $P_3$,
the quaternion group $Q_8$ of order 8, has trivial Schur multiplier and that
the Sylow-3 subgroup $(C_3 \times C_3)\rtimes C_3$ of $P_3$ has Schur
multiplier $C_3$. The former holds because every abelian subgroup of
$Q_8$ is cyclic, and one can show that all such group have trivial
Schur multiplier. To show
the latter, one can, for example, use the facts about the Schur multiplier
of a semi-direct product that are proven in \cite{evens}. 

\subsection{Representations with projective image isomorphic to $P_1$}  
\par We count the number of Artin representations $\rho$ unramified
outside $p\infty$ with projective
image isomorphic to $P_1$, i.e., $\Eva(F_{\rho}/\Q)\cong P_1$. Let $L$ denote
the fixed field of $C_3 \times C_3 \leq P_1$. Since $L$ is unramified
outside $p \infty$, it is the unique degree four
subfield of $\Q(\zeta_p)$. Writing $[E_{\rho}:F_{\rho}]=n$, we
have the following tower of fields.

\begin{figure}[h]
\begin{center}
  \setlength{\unitlength}{2.5cm}
  
 \begin{picture}(1,1.85)(2.7,.07)
 \put(3.04,0){$\Q$}
    \put(3.04,.6){$L$}
    \put(3.04,1.2){$F_{\rho}$}
    \put(3.04, 1.8){$E_{\rho}$}
   
    \put(3.15,.3){$C_4$}
    \put(3.15,1.53){$C_n$}
    \put(3.15,.9){$C_3\times C_3$}
    \qbezier(3.08,.21)(3.08,.18)(3.08,.53)
    \qbezier(3.08,.8)(3.08,.8)(3.08,1.15)
    \qbezier(3.08, 1.4)(3.08,1.4)(3.08,1.74)
\end{picture}
\end{center}
\end{figure}

\begin{lemma}\label{l1}
Let $K$ be a subfield of $\Q(\zeta_p)$, and let $\p$ be the unique prime of $K$ above the rational
prime $p$. Then $K^{\p\infty}=H_K(\zeta_p)$. 
\end{lemma}
\begin{proof}
Certainly $H_K(\zeta_p)\subseteq K^{\p\infty}$. Suppose the containment is
strict. Then $K^{\p\infty}/H_K(\zeta_p)$ is totally ramified at the primes of
$H_K(\zeta_p)$ above $p$. It follows that $K^{\p\infty}/\Q(\zeta_p)$ is tamely
ramified at the prime $(1-\zeta_p)$ of $\Q(\zeta_p)$ above $p$. But such an extension cannot exist since
$\Q(\zeta_p)^{(1-\zeta_p)}=H_{\Q(\zeta_p)}$ \cite[Proposition 3.4]{jing}.
\end{proof}
We have $F:=F_{\rho} \subseteq L^{\p}$, where $\p$ is the unique prime of $L$
above $p$. Suppose that the ramification degree $e=e(F/L)$ of $\p$ in
$F$ is 3 or 9. Then by Lemma \ref{l1}, the ramification in $F/L$ must be
obtained by adjoining the unique degree 3 ($e=3$) or 9 ($e=9$) subfield of $\Q(\zeta_p)$ to
$L$. It follows that
$|P_1'|=[F:F\cap \Q^{ab}]=3$ ($e=3$) or $1$ ($e=9)$. This is a contradiction since $P_1'=C_3 \times
C_3$. Therefore, $F/L$ is unramified and $F\subseteq H_L$. 
\par We have the following lemma of Tate:

\begin{lemma}\textup{(Tate \cite[Theorem 5]{weight1})}\label{tate} 
For each prime $p$, let $I_p$ be the inertia group of $p$ in
$\Eva_{\Q_p}$, which we view as a fixed subgroup of $\Eva_{\Q}$. Let
$\tilde{\rho}:\Eva_{\Q} \to \PGL_n(\C)$ be a projective
representation. Suppose that for each $p$, there exists a lift
$\rho_p:\Eva_{\Q_p}\to \GL_n(\C)$ of $\tilde{\rho}|_{\Eva_{\Q_p}}$ with $\rho_p|_{I_{p}}$ trivial for all but finitely many $p$. Then there
exists a unique lift $\rho: \Eva_{\Q}\to \GL_n(\C)$ of $\tilde{\rho}$ such that
$\rho|_{I_p}=\rho_p|_{I_p}$ for all $p$. 
\end{lemma}
Returning to our notation before Lemma \ref{tate}, since $F_{\rho}/\Q$
is unramified outside $p$, for all primes $l \neq p$,
$\tilde{\rho}|_{\Eva_{\Q_l}}$ factors through a cyclic extension of
$\Q_l$. The Schur multiplier $H_2(G, \Z)$ of a
cyclic group $G$ is trivial, so any projective representation of $G$
lifts to an ordinary representation of $G$ (see \cite[Chapter 20]{huppert}, in particular Theorem 20.8, part c). We may thus lift $\tilde{\rho}|_{\Eva_{\Q_l}}$ to a map
$\rho_l:\Eva_{\Q_l} \to \GL_3(\C)$ such that the fixed
field of the kernel of $\tilde{\rho}|_{\Eva_{\Q_l}}$ is the same as the
fixed field of the kernel of $\rho_l$; in particular, both fixed
fields are unramified at $l$. 
\par Since $F_{\rho}/L$ is
unramified, $F_{\rho}\Q(\zeta_p)/\Q(\zeta_p)$ is unramified. By class
field theory, the principal ideal $(1-\zeta_p)$ splits completely in
$F_{\rho}\Q(\zeta_p)$. It follows that the ideal $\p$ of $L$ splits completely in $F_{\rho}$. Thus the kernel of
the fixed field of $\tilde{\rho}|_{\Eva_{\Q_p}}$ is a totally tamely ramified
extension of $\Q_p$, so cyclic. We may therefore lift
$\tilde{\rho}|_{\Eva_{\Q_p}}$ to a map $\rho_p$ such that the kernel of
the fixed field of $\rho_p$ is the same as the kernel of the fixed field
of $\tilde{\rho}|_{\Eva_{\Q_p}}$.
\par Putting this local data together, Lemma \ref{tate} now tells us
that there exists a lift $\rho$ of $\tilde{\rho}$ such that the
ramification in $F_{\rho}/\Q$ is the same as the ramification in
$E_{\rho}/\Q$ (i.e., $E_{\rho}$ is unramified at all primes $l \neq p$
and the ramification degree at $p$ in $E_{\rho}/\Q$ is the same as the
ramification degree at $p$ in $F_{\rho}/\Q$). 
\par We will need the following theorem:

\begin{theorem}\textup{(Fr\"{o}hlich \cite[page 31; page 43, Corollary
    2]{frohl})}\label{thf}
Let $L/K$ be Galois and let $N$ be a central extension of $L/K$. Then
there is a surjective map
\[
 H_2(\Eva(L/K),\Z) \twoheadrightarrow
\Eva\big(N/L(N \cap K^{ab})\big),
\]
and there exists a central extension $N'$ of $L/K$ such that 
\[
 H_2(\Eva(L/K),\Z)  \to \Eva\big(N'/L(N' \cap K^{ab})\big) 
\]
is an isomorphism. Moreover, if $K=\Q$, one can take $N'$ to be ramified at the same
rational primes that $L$ is ramified at. 
\end{theorem}

\begin{remark}
The Schur multiplier $H_2(G, \Z)$ of a group
$G$ is intimately related to the study of central extensions and projective
representations of $G$. See, for example, \cite[Chapter 20]{huppert}.
\end{remark}

\begin{lemma}\label{c1}
For a prime $p \equiv 1 \pmod{4}$, let $L \subseteq \Q(\zeta_p)$ be
the unique field
with $[L:\Q]=4$. Suppose that $L$ admits an unramified $C_3\times
C_3$ extension $F$. Then $F/\Q$ has a unique degree three unramified
central extension. 
\end{lemma}
\begin{proof}
We first show the existence of such an extension. Let $\tilde{\rho}$
be a projective representation of $\Eva(F/\Q)$ with image isomorphic to $P_1$. Let $\rho$ be a lift of
$\tilde{\rho}$ such that the image of $\rho|I_p$ has order 4, the
existence of which is guaranteed by Lemma \ref{tate}. Let $N$ be the
fixed field of $\ker \rho$. Since the ramification degree of $p$ in
$L$ is 4, $N/L$ is unramified (recall we have assumed that $F/L$ is
unramified) and $N \cap F\Q^{ab}=F$. By Theorem \ref{thf}, we must
have $[N:N\cap F\Q^{ab}]=[N:F]\leq 3$. We cannot have $N=F$ because the
group $P_1$ does not have any irreducible three-dimensional
representations, so $[N:F]=3$. Thus $N$ gives the desired extension. 
\par Now for uniqueness. Suppose that $N_1$ and $N_2$ are distinct
non-trivial unramified extensions of $F/\Q$. We have
$H_2(\Eva(F/\Q),\Z)=C_3$ and $F(N_i\cap \Q^{ab}) \subseteq N_i \cap F
\Q^{ab}$, so it follows from Theorem \ref{thf} that $C_3
\twoheadrightarrow \Eva(N_i/N_i \cap F \Q^{ab}), i=1,2$. Since $N_i/F$
and $F/L$ are unramified, we have $N_i \cap \nobreak F \Q^{ab}=F$. Suppose $N_1
\neq N_2$. Then $N_1N_2/F$ is an unramified central extension of
$F/\Q$ with $N_1N_2 \cap F \Q^{ab}=F$ and $[N_1N_2:F]=9$,
contradicting Theorem \ref{thf}. 
\end{proof}

\begin{lemma}\label{maxcent}
Let $F/\Q$ be a finite Galois extension with group $G$, and suppose
$H_2(G, \Z) \cong C_3$. For a fixed number field $M \subset
\Q(\zeta_{p^{\infty}})$, $FM/\Q$ admits a maximum of three distinct
non-trivial central extensions $E_i$ with $E_i \cap \Q^{ab}=M$ and
$E_i/\Q$ unramified outside $p \infty$.
\end{lemma}

\begin{proof}
Let $E_i$ be fields as in the statement of the lemma. By Theorem
\ref{thf}, for each $i$, we have $[E_i:FM]=3$. The compositum $\prod
E_i$ is a central extension of $FM/\Q$. By Theorem \ref{thf}, there
exists $M' \subset \Q(\zeta_{p^{\infty}})$ containing $M$ such that
$|\Eva(\prod E_i/FM')|\leq 3$. This means $\prod E_i/FM$ has rank at
most two, i.e., its Galois group is $C_3$ or $C_3 \times C_3$. Such a
field extension has at most (in the $C_3 \times C_3$ case) four
strictly intermediate extensions-- three $E_i$'s and $FM'$, (and we
must have $[M':M]=3$). 

\end{proof}

\par We now come to the main result of this section. 
\begin{proposition} \label{mainth} Let $\rho:\Eva_{\Q}\to \GL_3(\C)$ be an Artin representation
unramified outside $p\infty$ with projective image isomorphic to $P_1$. Let $F$ be the
fixed field of $\ker{\pi \circ \rho}$. Let $\Q(\zeta_{p,3})$ denote the maximal
subfield of $\Q(\zeta_p)$ of $3$-power degree over $L$, and let
$\Q(\zeta_{p,{3'}})$ be the composite of all subfields of
$\Q(\zeta_{p^{\infty}})$ with degree prime-to-$3$ over $L$. Then the fixed field of $\ker{\rho}$ either 
\begin{itemize}
\item[\textup{i)}] lies between $N$ and
$N\Q(\zeta_{p,{3'}})$, where $N$ is the unique non-trivial unramified
central extension of $F/\Q$, and $[N:F]=3$, or 
\item[\textup{ii)}] lies between $E_3$
and $E_3\Q(\zeta_{p,{3'}})$, where $E_3 \nsubseteq F\Q(\zeta_{p,3})$
is one of two possible degree 3 ramified
extensions of $FM_3$ that is central over $F/\Q$, where $M_3$ is a
subfield of $ \Q(\zeta_{p,3})$.
\end{itemize}
\end{proposition}
\begin{proof}
The existence and uniqueness of $N$ is given by Lemma \ref{c1}. For
the entirety of the proof, we consider representations $\rho$ for
which $F_{\rho}=F$, with $F$ fixed. Let $E$ be a candidate for the fixed field of $\ker{\rho}$. By Theorem \ref{thf}, there exists $M \subset
\Q(\zeta_{p^{\infty}})$ such that $[E:FM]\leq 3$, and $E \cap
\Q^{ab}=M$. We claim that $E\neq FM$. If $E$ were contained in
$FM$ with $[E:F]=2^km, m$ odd, then we would have $\Eva(E/\Q)\cong
(C_3 \times C_3) \rtimes (C_{2^k}\times C_m)$, with $C_{2^k}\times C_m$
acting on $C_3 \times C_3$ through its $C_4$ quotient. Thus $C_3
\times C_3 \times C_m$ is an abelian normal subgroup of $\Eva(E/\Q)$
of index $2^k$. By \cite[Corollary 53.18]{CR}, the degree of every
irreducible representation of $\Eva(E/\Q)$ divides $2^k$. In
particular, $\Eva(E/\Q)$ has no irreducible three-dimensional
representations, a contradiction.

\par The proof will be complete upon establishing the following two
claims.
\newline

\textbf{Claim I}: If $E/FM$ is unramified, then we are in Case i) of the
proposition.
\newline
\par
\textbf{Claim II}: If $E/FM$ is ramified, then we are in Case ii) of the
proposition. 
\newline
\par
\underline{Proof of Claim I.}
We are assuming that $E/FM$ is unramified. Being a composite of two central extensions of $F$, $NM$
is a central extension of $F/\Q$ as well (see Figure \ref{fig}).

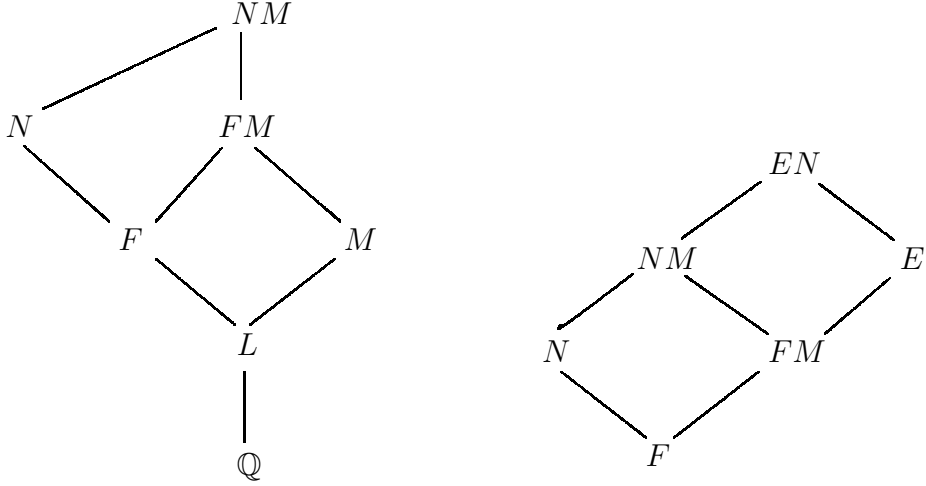
\begin{figure}[h]
\begin{center}
  \setlength{\unitlength}{2.5cm}
  
 \begin{picture}(1,2.5)(3.2,0)
 \put(3.03,0){$\Q$}
    \put(3.03,.64){$L$}
    \put(3.6,1.2){$M$}
    \put(2.4,1.2){$F$}
    \put(2.93, 1.8){$FM$}
     \put(3,2.4){$NM$}
    \put(1.8,1.8){$N$}
   
   \qbezier(3.05,2.00)(3.05,2.00)(3.05,2.35)
    \qbezier(3.07,.18)(3.07,.18)(3.07,.55)
    \qbezier(3.02,.8)(3.02,.8)(2.6,1.15)
     \qbezier(3.1,.8)(3.1,.8)(3.55, 1.15)
    \qbezier(2.6, 1.35)(2.6,1.35)(2.95,1.74)
    \qbezier (3.58, 1.35)(3.58,1.35)(3.13, 1.74)
    \qbezier(2.35, 1.35) (2.35, 1.35)(1.9, 1.75)
    \qbezier(2.0, 1.95)(2.0, 1.95)(2.92,2.38)
\end{picture}
\begin{picture}(1,2)(5,0)
\put(5.95,.05){$F$}
\put(5.4,.6){$N$}
\put(6.6,.6){$FM$}
\put(5.9,1.1){$NM$}
\put(6.6,1.6){$EN$}
\put(7.3, 1.1){$E$}
\qbezier(6.1,.2)(6.1,.2)(6.55,.55)
\qbezier(5.95,.2)(5.95,.2)(5.5,.55)
\qbezier(5.5,.8)(5.4,.7)(5.88,1.07)
\qbezier(6.16,1.06)(6.16,1.06)(6.6,.75)
\qbezier(6.14,1.26)(6.14,1.26)(6.56,1.56)
\qbezier(6.9,.75)(6.9,.75)(7.25,1.05)
\qbezier(6.87,1.55)(6.87,1.55)(7.27,1.25)
\end{picture}
\end{center}
\caption{Field diagrams for proof of Proposition \ref{mainth}.} \label{fig}
\end{figure}

If $E\neq NM$, then $EN$ is an unramified central extension of
$FM$. Since $F/L$ is unramified, if $EN/FM$ is unramified, then
$EN\cap FM \Q^{ab}=FM$ with $[EN:FM]=9$, which contradicts Theorem
\ref{thf}. So $E=NM$, and we have only to show that $M \subseteq
\Q(\zeta_{p,3'})$. Suppose that $3 \mid [M:\Q]$. Then $NM/F$,
having two intermediate subfields of degree $3$ over $F$, is not
cyclic. Thus we must have $M\subseteq \Q(\zeta_{p,{3'}})$.
\newline
\par
\underline{Proof of Claim II.} Let $M_3=M \cap \Q(\zeta_{p,3})$ and
$M_{3'}=M \cap \Q(\zeta_{p,3'})$, so $M=M_3M_{3'}$. Since $E/FM_3$ is cyclic
(recall that $E/F$ is cyclic) and $[E:FM_3]$ is divisible by $3$ but
not by $3^2$, there exists a degree 3 extension $E_3$
of $FM_3$ such that $E=E_3M$. Note that $E_3M=E_3M_{3'}$. An elementary ramification argument
shows that $E/FM_3$, and thus $E_3/FM_3$, is totally ramified.
 Define
\[\Sigma_{M_3}=\{\textrm{Central extensions }J \textrm{ of } F/\Q
\textrm{ containing $M_3$}: J/FM_3 \textrm{ is
  ramified at $p$ and only at $p$, } \]
\[[J:FM_3]=3, \textrm{ and } J\cap
\Q^{ab}=M_3 \}.
\]
We are assuming $E_3/FM_3$ is ramified, so $E_3 \in \Sigma_{M_3}$. From Lemma
\ref{maxcent}, we have $\#\Sigma_{M_3} \leq 3$. To complete the proof
of Claim II, we need to show that $\#\Sigma_{M_3}
\leq 2$. From the proof of Lemma \ref{maxcent}, the compositum $J'$ of all
elements of $\Sigma_{M_3}$ is contained in a $C_3 \times C_3$ extension
of $FM_3$. We may assume $\#\Sigma_{M_3}=3$, and thus that $J'$ is equal to a $C_3
\times C_3$ extension of $FM_{3}$. The proof of Lemma \ref{maxcent} also
tells us that $FM'\subset J'$, where $M'\subset
\Q(\zeta_{p^{\infty}})$ and $[M':M_3]=3$. Since $F/L$ is unramified,
$FM_3/M_3$ is unramified as well, from which it follows that $FM'/FM_3$ is
ramified, as $M'/\Q$ is totally ramified. If $\#\Sigma_{M_3}=3$, then all four
intermediate subfields of $J'/FM_3$ are ramified over $FM_3$, and thus
$J'/FM_3$
is totally
ramified. This is a contradiction, however, since the corresponding
extensions of local fields at primes above $p$ would be totally tamely
ramified but not cyclic. 
\end{proof}

The next goal is to count the number of representations $\Eva_{\Q}\to
\GL_3(\C)$ with projective image isomorphic to $P_1$ and given prime
power conductor. We must have
$p\equiv 1 \pmod{4}$, and we will also assume (as is necessary for such
a representation to exist), that the $3$-rank $k$ of the
class group $\Cl(L)$ of $L$ is at least 2. Then $L$ admits $r_{3,2}(k)$
unramified extensions $F$ with $\Eva(F/L)=C_3 \times C_3$. 
\par Let $J=\Eva(N/\Q)$, where $N$ is as in Case i) of Proposition \ref{mainth}. Since
the 3-Sylow subgroup $\Eva(N/L):=A$ of $J$ is normal in
$J$, the group $J$ splits as
$A\rtimes C_4$. If $A$ is abelian, then $J$ is metabelian (i.e., the
second commutator group of $G$ is trivial), which gives a
contradiction since metabelian groups are monomial \cite[Theorem 52.2]{CR}, and $\rho$ is primitive. Suppose that $A$ is isomorphic to the
non-trivial semi-direct product $C_9 \rtimes C_3$. A generator of $C_4$ acts on $J/Z(J)$
as an automorphism of order 4 and thus must act on $A$ as an
automorphism of order a multiple of 4. But an exercise in
group theory shows that $|\Aut(C_9 \rtimes \nobreak
C_3)|=54$, so no such automorphism exists. Thus $A$ must be the
other non-abelian group of order 27, $(C_3 \times C_3)\rtimes
C_3$. A computer calculation \cite{GAP4} shows the group $J$ has 8 irreducible three-dimensional
representations. Let $S$ denote the set of these representations. As an example, we determine the conductors of these
representations-- that is, the conductors of the representations of $\Eva_{\Q}$ for which
the fixed field of the kernel is $N$. Recall the definition of the Artin conductor
$N(\rho)$ of a representation $\rho$ from
Section \ref{intro}. We have $J_{0}=C_4$, $J_i=0, i \geq 1$. Using
\cite{GAP4}, one
finds that for 2 of the 8 aforementioned representations, $\dim V^{J_0}=0$, and for the other 6, 
dim$V^{J_0}=1$, where $V\cong \C^3$ is the representation space. So we
get two representations with conductor $p^3$ and 6 with conductor $p^2$.
\par Suppose now that we are in Case i) of Proposition \ref{mainth},
so the fixed field of the kernel of a
representation $\rho$ is of the form $NM$, with  $M\subset
\Q(\zeta_{p,{3'}})$. The irreducible
three-dimensional representations of $\Eva(NM/\Q)$ are given by $\psi
\otimes \chi$, where $\psi$ is an element of $S$ and $\chi$ is a character of $\Eva(M/L)$, of which there
are $[M:L]$ (by abuse of notation, we are allowing $\rho$ and $\chi$
to also denote their respective lifted representations to the group
$\Eva(NM/\Q)$; furthermore, we are identifying characters of $G(M/L)$
with characters of $G(M/\Q)$ modulo characters of $G(L/\Q)$). To see this, note that the sum of the squares of the
degrees of the irreducible representations of $\Eva(NM/\Q)$ is equal to
the order of the group. Consider the irreducible representations of $\Eva(NM/\Q)$ the
form $\psi \otimes \chi$, where $\psi$ ranges over all
irreducible representations of $\Eva(N/\Q)$ and $\chi$ ranges over all irreducible
representations of $\Eva(M/L)$. The sum of the squares of the degrees
of these representations is the order of $\Eva(NM/\Q)$. If we can show
that these representations are distinct as $\psi$ and $\chi$ vary,
then it will follow
that these are all the irreducible representations of
$\Eva(NM/\Q)$. This amounts to showing that for distinct
representations $\psi$ and $\psi'$ of $G(N/\Q)$ we cannot have $\psi
\otimes \chi \cong \psi'$. One can see this by considering the restrictions of
both representations to $G(NM/N)$. 
\par Thus we find that the number of irreducible representations of $\rho$
of dimension three with $\ker{\rho}$ having fixed field $NM$ is
$8[M:L]$. 
\par Continuing with Case i) of Proposition \ref{mainth}, we now fix a $C_3 \times C_3$ extension $F$ of $L$ and count the
number of $\rho$ for which $N(\rho)= p^m$ and $F_{\rho}=F$. Write
$[M:\Q]=4cp^{n}$, with $c\mid \frac{p-1}{4}$, so \newline $[M:L]=p^nc$. Let
$G=\Eva(NM/\Q)$, and let $G_i$, with order $g_i$,
denote the lower
ramification groups of $NM/\Q$. (Technically, we mean the lower ramification
groups of the extension $NM_{\p}/\Q_p$ of local fields where $\p$ is a
prime of $NM$ lying
above $p$. Since $NM/\Q$ is Galois and ramified only at $p$, there is
no ambiguity created by this language, and we use it henceforth when
we have a Galois extension of $\Q$ ramified at a single rational prime). As soon as $M \nsubseteq N$, we get
non-trivial central elements in $G_i$ whenever $G_i\neq 1$. Such
elements act by scalars and thus do not fix any non-zero subspace of $V$. So
for such $\rho$, $v_p(N(\rho))=\sum_{i\geq 0, g_i\neq 1}\frac{g_i}{g_0}\cdot 3$. We
must compute the $g_i$. We know $g_0=4cp^n, g_1=p^n$. 
\begin{lemma}\label{lemram}
Let $E/\Q$ be a Galois extension unramified outside $p\infty$ and
tamely ramified at $p$. Let $M$ be subfield of $\Q(\zeta_{p^n})$ that
is not contained in $\Q(\zeta_{p^{n-1}})$. Then for $i\geq 1$, the $i$th
lower ramification groups of $\Eva(EM/\Q)$ have the same order as the
respective lower ramification groups of $\Q(\zeta_{p^{n}})/\Q$. 
\end{lemma}

\begin{proof}
Let $G=\Eva(EM/\Q),
G'=\Eva(\Q(\zeta_{p^{n}})/\Q)$, with upper ramification groups $G^i$
and $G'^i$, respectively, and let
$H=\Eva(EM/M), H'=\Eva(\Q(\zeta_{p^{n}})/M)$. Since $G^i$ is a $p$-group, for all
$i>0$,
\begin{equation}
(G/H)^i=G^iH/H=G^i/G^i\cap H=G^i \textrm{ and} \label{upper1}
\end{equation}
\begin{equation}
(G'/H')^i=G'^iH'/H'=G'^i/G'^i\cap H' =G'^i, \label{upper2} 
\end{equation}
where the first equality in both lines is a
property of upper ramification groups (see \cite[Chapter IV]{MR554237} for
details about the upper and lower ramification groups). The order of
the lower ramification groups is determined by the order of the upper ramification
groups. Since $(G/H)^i=(G'/H')^i$, \eqref{upper1} and
\eqref{upper2} imply
that $|G_i|=|G'_i|$ for all $i\geq 1$ .
\end{proof}
Let $G':=\Eva\big(\Q(\zeta_{p^{n+1}})/\Q\big)$. As shown in \cite[Chapter IV.4]{MR554237},
for any $0 \leq k \leq n$,

\[
|G'_i|=p^k \textrm{\quad for\quad } p^{n-k}\leq i \leq p^{n-k+1}-1, 
  \quad i \in \Z.
\]

By Lemma \ref{lemram}, $|G_i|=|G'_i|$ for all $i\geq 1$. So $|G_i|= p^k$ for $p^{n-k}\leq i \leq p^{n-k+1}-1$ and $0\leq k \leq n$,
and $|G_i|=1$ for $i\geq p^n$. Thus if $\ker \rho$ has fixed field
$NM$, 
\[
v_p\big(N(\rho)\big)=
3\left(1+\frac{p-1}{4c}+\frac{p(p-1)}{4cp}+\cdots + \frac{p^{n-1}(p-1)}{4cp^{n-1}}\right)=3+3n\frac{p-1}{4c},
\]
and there are $8p^nc$ representations with kernel fixing $NM$. 
\par We are now in a position to estimate the number $a_{m,i}$ of $\rho$
for which we are in Case i) of Proposition \ref{mainth} and for which $N(\rho)= p^m$ and $F_{\rho}=F$. If $m \equiv 1$ or $2
\pmod{3}$, then $a_{m,i}=0$. Suppose $m \equiv 0 \pmod{3}$. Let
$d=\gcd(\frac{p-1}{4},k)$, where $m=3k+3$. Then
\[
a_{m,i}=8p^k\frac{p-1}{4}+\cdots +
8p^{\frac{k}{d}}\frac{p-1}{4d}=8\sum_{j\mid d}p^{\frac{k}{j}}\frac{p-1}{4j}.
\]
The first term in the sum, for example, corresponds to the case
$c=\frac{p-1}{4},n=k$, in which case there exist $8p^k\frac{p-1}{4}$
representations $\rho$ with
$v_p\big(N(\rho)\big)=3+3n\frac{p-1}{4c}=3+3k$.

For the total number $A_{m,F,i}$ of representations
$\rho$ in Case i) of Proposition \ref{mainth} with
$F_{\rho}=F$ and $N(\rho) \mid p^m$, we have 
\[
A_{m,F,i} \sim \sum_{0 \leq k \leq \frac{m-3}{3}}2p^{k+1} \sim 2p^{\frac{m}{3}},
\]
where we are viewing $A_{m,F,i}$ as a function of $p$.

\par Suppose now that we are in Case ii) of Proposition
\ref{mainth}. We can write the fixed field of $\ker{\rho}$ as $E_3M_{3'}$,
with $M_{3'}\subset \Q(\zeta_{p,3'})$ and $E_3$ a degree 3 ramified
extension of $FM_3$ not contained in $F\Q(\zeta_{p,3})$, with $M_3 \subseteq \nobreak \Q(\zeta_{p,3})$. Let
$[M_3:L]=3^a$. The group $\Eva(E_3/\Q)$ splits as  $B \rtimes
\nobreak C_4$, where $B=\Eva(E_3/L)$ is a $3$-group. \newline
\par
\textbf{Claim}: $B\cong
(C_{3^{a+1}}\times C_3)\rtimes C_3$, where $Z(B)$ is the first
$C_{3^{a+1}}$ factor, which is also the center of $\Eva(E_3/\Q)$ (a generator of
the last $C_3$ factor acts on $C_{3^{a+1}}\times C_3$ by fixing the
$C_{3^{a+1}}$ factor and mapping $(0,1) \mapsto (3^a,1)$). \newline
\par
\underline{Proof of Claim.} The same
reasoning that showed $A\cong (C_3 \times C_3) \rtimes C_3$ in the
discussion immediately following the proof of Proposition \ref{mainth}
shows that $\Eva(E_3/M_3)\cong (C_3 \times C_3) \rtimes C_3$. Since
$E_3/F$ is cyclic, it follows
that if $F'$ is a field strictly between $L$ and $F$, then
$\Eva(E_3/F')$ has rank two (see Figure \ref{fig'}). The group $\Eva(E_3/F')$ is abelian since
its center has index at most three. Thus $\Eva(E_3/F')\cong
C_{3^{a+1}}\times C_3$, and it is normal in $B$ since $F'/L$ is
Galois. Since $\Eva(E_3/M_3)\cong (C_3 \times C_3) \rtimes C_3$, it has
a subgroup $C$ of order three not contained in $\Eva(E_3/F')$, and thus
$B$ splits as $\Eva(E_3/F') \rtimes C$, as claimed.
\newline

\begin{figure}[h]
\begin{center}
  \setlength{\unitlength}{2.5cm}
  
 \begin{picture}(1,2.5)(2.5,0)
 \put(3.03,0){$\Q$}
    \put(3.03,.64){$L$}
    \put(3.6,1.2){$M_3$}
    \put(2.4,1.2){$F$}
    \put(2.93, 1.8){$FM_3$}
     \put(3,2.2){$E_3$}
    \put(3.9,1.65){$M_3M_{3'}$}
    \put(3.4,2.5){$E_3M_{3'}$}
    \put(2.76,.98){$F'$}
   
   \qbezier(3.05,1.98)(3.05,1.98)(3.05,2.15)
    \qbezier(3.07,.18)(3.07,.18)(3.07,.55)
    \qbezier(3.02,.8)(3.02,.8)(2.85,.95)
    \qbezier (2.72,1.05)(2.72,1.05)(2.56,1.18)
     \qbezier(3.1,.8)(3.1,.8)(3.55, 1.15)
    \qbezier(2.6, 1.35)(2.6,1.35)(2.95,1.74)
    \qbezier (3.53, 1.35)(3.53,1.35)(3.13, 1.74)
    \qbezier(3.75, 1.35) (3.75, 1.35)(3.93, 1.56)
    \qbezier(4.03, 1.8)(4.03, 1.8)(3.56,2.41)
    \qbezier (3.17,2.3)(3.17,2.3)(3.37,2.5)

\end{picture}
\end{center}
\caption{}\label{fig'}
\end{figure}

\par We would now like to count the number of irreducible
representations of $\Eva(E_3/\Q)\cong B \rtimes C_4$. While counting representations
from Case i) of Proposition \ref{mainth}, we saw that \newline $((C_3 \times
C_3)\rtimes C_3)\rtimes C_4\cong((3^aC_{3^{a+1}}\times C_3)\rtimes
C_3)\rtimes C_4:=B'\rtimes C_4$ has 8
irreducible three-dimensional representations. If $\psi$ is an
irreducible representation of $B'\rtimes C_4$, and $\psi$ maps the
generator $3^a \in 3^a\Z/3^{a+1}\Z$ of the
$3^aC_{3^{a+1}}$ factor to the scalar $\alpha \in \GL_3(\C)$, then
mapping the generator $1 \in \Z/3^{a+1}\Z$ of the $C_{3^{a+1}}$ factor
to any of $3^a$th roots of $\alpha$ gives a well
defined irreducible representation of $B \rtimes C_4$, and distinct
choices of $3^a$th roots of $\alpha$ give distinct
representations. Using the fact that the sum of the squares of the
degrees of the irreducible representations is the order of the group,
we see that we obtain \textit{all} irreducible representations of $B
\rtimes C_4$ in this way. It follows that the number of irreducible
three-dimensional representations of $\Eva(E_3/\Q)$ is $8\cdot
3^a$. Recall that in Case ii), the fixed field of the kernel of our given
representation $\rho$ is of the form $E_3M_{3'}$, where $M_{3'} \subseteq
\Q(\zeta_{p,3'})$. The number of irreducible
three-dimensional representations of $\Eva(E_3M_{3'}/\Q)$ is thus $8\cdot
3^a \cdot [M_3M_{3'}:M_3]$. 
\par Write $[E_3M_{3'}:\Q]=108p^nc$, with $(c,p)=1$. Then our number $8
\cdot 3^a \cdot [M_3M_{3'}:M_3]$ from above is equal to $8p^nc$. Let $G_i$, with order $g_i$, be the lower
ramification groups of $\Eva(E_3M_{3'}/\Q)$. For all $i\geq 0$,
$G_i$ has non-trivial central elements, so $\textrm{dim}V/V^{G_i}=3$ whenever
$g_i\neq 1$. So again we obtain
\begin{equation}
v_p\left(N(\rho)\right)= \sum_{i \geq 0, g_i \neq 1} \frac{g_i}{g_0}\cdot 3, \label{eqN}
\end{equation}
and we have to compute the $g_i$. We have $g_0=12p^nc$. It follows from
Lemma
\ref{lemram} that the $g_i$ for Case ii) are
the same as those for Case i). We are thus able to compute, analogously to Case i),
\[
v_p\big(N(\rho)\big)=
3+3n\frac{p-1}{12c}.
\]
In the same way that we bounded $a_{m,i}$, we can now bound the number $a_{m,ii}$ of $\rho$ for which we are in
Case ii) of Proposition \ref{mainth} and for which $N(\rho)=p^m$ and
$F_{\rho}=F$.  By \eqref{eqN}, $a_{m,ii}=0$
unless $m \equiv 0 \pmod{3}$, so write $a_{m,ii}=3k+3$. A similar series
of computations to those in Case i) gives
\[
a_{m,ii}=8p^k\frac{p-1}{4}+\cdots +
8p^{\frac{k}{d}}\frac{p-1}{12d}=8\sum_{j\mid
  d}p^{\frac{k}{j}}\frac{p-1}{12j} \textrm{ ,} \quad d=\gcd \left(\frac{p-1}{4},k \right).
\]
When we are in Case ii), for a fixed $M_3$ and $F$, there are two
possible degree 3 extensions of $FM_3$ that are candidates for the
kernel of the fixed field of the representation $\rho$, so we obtain
\[
A_{m,F,ii} \sim \sum_{0 \leq k \leq \frac{m-3}{3}}
  \frac{4}{3}p^{k+1}\sim \frac{4}{3}p^{\frac{m}{3}}.
\]

\par To bound the total number $B_{m}$ of primitive representations $\rho$ with
$N(\rho)\mid p^m$ and \mbox{$\Imm(\pi \circ \rho)=P_1$}, we need to multiply $A_{m,F,i}+A_{m,F,ii}$ by the
number of unramified $C_3
\times C_3$ extensions of $L$, which, as previously discussed, is
$r_{3,2}(k)$, where $k$ is the $3$-rank of $\Cl(L)$. From
\cite[page 95]{moonmonomial}, the $3$-part of $h_L$ is less than $\frac{22.2p^3}{\pi^4}$, so $k \leq
\log_3(\frac{22.2p^3}{\pi^4}):=R$. 
\par We introduce the following (non-standard) notation. Suppose that
$A$ and $B$ are both positive, real-valued functions of $n$.
\[ \textrm{We write }
A\lessapprox B \textrm{ to mean }
\limsup_{n \to \infty}\frac{A(n)}{B(n)}\leq 1. 
\]
Using this notation,
we obtain:
\begin{theorem}\label{t4}
Notation as above, $$
B_m \lessapprox (A_{m,F,i}+A_{m,F,ii})r_{3,2}(R)\sim \frac{4928.4}{3^5 \pi^8}p^{\frac{m}{3}+9}.$$
                               
\end{theorem}
\subsection{Representations with projective image isomorphic to $P_3$}\label{4}
Recall that $P_3:= \big( C_3 \times C_3 \big) \rtimes \SL_2(\F_3)$. Using the notation from the previous section, we look to count the
number of $\rho$ with $\Eva(F_{\rho}/\Q)\cong P_3$ and given conductor. Our main tool will
be the following proposition.
\begin{proposition}\label{thp3}
Let $\rho:\Eva_{\Q}\to \GL_3(\C)$ be an Artin representation
unramified outside $p\infty$ with projective image isomorphic to
$P_3$. Let $F$ be the fixed field of $\ker \pi \circ \rho$. If $M \subseteq
\Q(\zeta_{p^{\infty}})$, then there are at most three degree 3 central
extensions of $F_{\rho}M/\Q$ ramified over $F_{\rho}M$ and at most one
that is unramified over $F_{\rho}M$. There exists
 $M \subseteq \Q(\zeta_{p^{\infty}})$ such that of
the four possible degree 3 central extensions of $F_{\rho}M/\Q$, the
fixed field of $\ker \rho$ must be the unramified extension. 

\end{proposition}

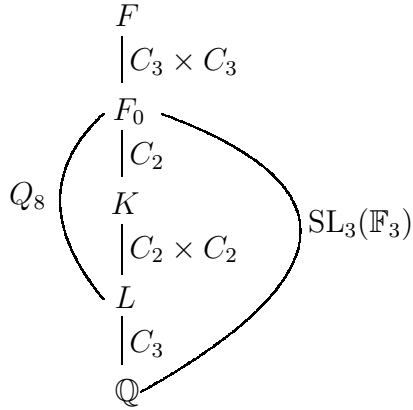
\begin{figure}[h]

  \setlength{\unitlength}{2.5cm}
  \begin{center}
 \begin{picture}(1,2.1)(2.9,0)
 \put(3.02,0){$\Q$}
    \put(3.02,.5){$L$}
    \put(3,1){$K$}
    \put(3.01,1.5){$F_0$}
    \put(3.02, 2){$F$}
    \put (3.1, .265){$C_3$}
    \put (3.1, .765){$C_2 \times C_2$}
    \put (3.1, 1.265){$C_2$}
     \put (3.1, 1.765){$C_3 \times C_3$}
     \put (2.46, 1.05){$Q_8$}
      \put(4.05, .9){$\SL_3(\F_3)$}
     
   \qbezier(2.96, .545)(2.5, 1.1)(2.96, 1.53)
   \qbezier(3.06,.2)(3.06,.2)(3.06,.45)
    \qbezier(3.06,.68)(3.06,.68)(3.06,.94)
    \qbezier(3.06, 1.2)(3.06, 1.2)(3.06, 1.44)
     \qbezier(3.06, 1.7)(3.06,1.7)(3.06, 1.94)
    \qbezier(3.16, .05)(4.8,.97)(3.27 ,1.53)

\end{picture}
\end{center}
\caption{Field diagram for Section \ref{4}.}
\end{figure}
\begin{proof} We have $H_2(P_3,
\Z)\cong C_3$. Let $E$ be a candidate for the fixed field of $\ker
\rho$ and let $M=E\cap \Q^{ab}$. By Lemma \ref{maxcent} and its proof, $FM/\Q$ admits a maximum of three distinct non-trivial
central extensions $E_i$ with $E_i \cap \Q^{ab}=M$, each of which has
degree 3 over $FM$, in addition to the degree 3 central extension $FM'$ with $M' \subseteq
\Q(\zeta_{p^{\infty}})$. 
\par The group $\Eva(FM'/\Q)$ has an abelian
normal subgroup that is not cyclic, and it follows from
\cite[Corollary 6.13]{MR2270898 } that such a group has no faithful primitive representations, so $FM'$ cannot be the fixed field
of $\ker \rho$ (the same argument shows $FM''$ cannot be the fixed field of
$\ker \rho$ for any $M'' \subset \Q(\zeta_{p^{\infty}})$). Since we cannot have a totally
ramified $C_3\times C_3$ extension of $FM$ that is Galois over $\Q$,
we must have either at least one of the $E_i's$ be unramified or have
$FM'/FM$ unramified. If $FM'/FM$ were unramified, ramification considerations force $e: =e(F/F_0)$ to
be 3 or 9, where $F_0$ is the fixed field of the $C_3 \times C_3$
normal subgroup of $P_3$. This is impossible, however, since we would
have $e(F/\Q)=3e$, so the extension of $\Q_p$ given by completing $F$
at a prime above $p$ would have a $C_{3e}$ subgroup, and $P_3$ does not
have such a subgroup. Thus one of the $E_i$'s is unramified over
$FM$. If two of the $E_i$'s were unramified, the composite of these
fields would give a $C_3 \times C_3$ unramified central
extension of $FM$ that contains $FM'$, forcing $FM'/FM$ to be
unramified, which we just saw gives a
contradiction. 
\par Thus, $FM/\Q$ may have up to four
central extensions of degree 3, and at this point, three of them (all but $FM'$) are
possibilities for $E$. By Theorem
\ref{thf},
$E$ is a degree 3 central extension of $FM/\Q$ for some $M
\subset \Q(\zeta_{p^{\infty}})$ since $C_3$ surjects onto
$\Eva(E/FM)$. It remains to show that $E/FM$ must be
unramified. Suppose that $E/FM$ is ramified. Write $[FM:F]=3^{a}b$ with
$3 \nmid b$, and let $M_3 \subset \Q(\zeta_{p^{\infty}})$ be the field
satisfying $[FM_3:F]=3^a$. The extension $FM_3/F$ is totally ramified since $F/F_0$
is unramified. We have $\Eva(FM_3/\Q)\cong (C_3 \times C_3) \rtimes
(Q_8 \rtimes C_{3^{a+1}})$ with the $C_{3^{a+1}}$ factor acting
through its $C_3$ quotient. The $3$-part of the inertia group at
a prime above $p$ in $E$ is cyclic of order $3^{a+2}$. It follows that
$\Eva(E/\Q) \cong  (C_3 \times C_3) \rtimes
(Q_8 \rtimes C_{3^{a+2}})$ with the $C_{3^{a+2}}$ factor acting
through its $C_3$ quotient. The left $(C_3 \times C_3)$ factor is then
an
abelian normal subgroup of $\Eva(E/\Q)$, which contradicts the
primitivity of the representation. 
\end{proof}

\begin{remark}\label{E0}
Since $E/F$ is abelian (it is cyclic) and $E/FM$ is unramified, there exists a field
$E_0$ between $E$ and $F$ such that $E_0/F$ is unramified of degree 3
and $E/E_0$ is totally ramified. Since $FM_3/F$ is totally ramified, $E_0$ cannot be contained in
$FM_3$ (so $E_0 \nsubseteq FM$), and thus we have
$E=E_0M$. Since $E/F$ is cyclic, we also have that $[FM:F]$ is not
divisible by $3$. 
\end{remark}

\begin{remark}\emph{
Unlike in the $P_1$ case, we may not use Lemma \ref{tate} here to guarantee the existence of an
unramified degree 3 central extension of $F$ because the extension of $\Q_p$
given by completing at a prime above $p$ in $F$ is not necessary
cyclic, so we are not assured of a lift $\rho_p:\Eva_{\Q_p}\to
\GL_3(\C)$ of $\tilde{\rho}|_{\Eva_{\Q_p}}$ such that the fixed field
of the kernel of $\rho_p$ is unramified over the fixed field of the
kernel of $\tilde{\rho}|_{\Eva_{\Q_p}}$.}
\end{remark}
Maintaining the notation from Figure \ref{fig'}, let $L=F^{ab}$ be the degree 3 extension of $\Q$ contained in
$\Q(\zeta_p)$. The field $L$ then admits a $C_2 \times C_2$
extension, which is unramified by Lemma \ref{l1}. In the proof of
Proposition \ref{thp3} we saw that
$F/F_0$ must be unramified, so $e(F/\Q)=3\textrm{ or } 6$.

Continuing in the manner of the $P_1$ case, we now bound, for a fixed
$F/\Q$ with $\Eva(F/\Q)\cong P_3$, the number of $\rho$ for which
$N(\rho)=p^m$ and $F_{\rho}=F$. Let $E=E_{\rho}$, so
$E$ is of the form $E_0M$ for some $M \subseteq
\Q(\zeta_{p^{\infty}})$, where $E_0$ is as in Remark \ref{E0}. Given
$M$, by Proposition \ref{thp3} there is one possible choice for $E$. Write
$[M:L]=p^nc$, with $c\mid \frac{p-1}{3}$. We have $e(E/\Q)=3p^ncx$,
where $x=1$ or $2$. Using
Lemma \ref{lemram} to compute
analogously to the $P_1$ case, we find
\[
v_p(N_{\rho})=3+3n\frac{p-1}{3cx}. 
\]

An exhaustive search using \cite{GAP4} shows that there are 3 non-split central
extensions of $P_3$ of order $3|P_3|=648$, and using \cite{GAP4}, we
can check that each of them has 7 irreducible
three-dimensional representations. 
\begin{remark}
We are considering central extensions of $P_3$ by $C_3$ up to group
isomorphism, not up to isomorphism of central extensions, which are
classified by $H^2(P_3, C_3)$. 
\end{remark}
The three-dimensional irreducible representations of $\Eva(E/\Q)$ are
all the $\psi \otimes \chi$, as $\psi$ runs over the three-dimensional
irreducible representations of $\Eva(E_0/\Q)$ and $\chi$ runs over all
characters of $\Eva(E/E_0) \cong \Eva(M/L)$. Thus $\Eva(E/\Q)$ has $7p^nc$ irreducible
three-dimensional representations. We estimate the number $a_{m,F}$ of
representations $\rho$ with $N(\rho)=p^m$ and $F_{\rho}=F$. We must
have $3\mid m$ so write $m=3k+3$. Proceeding as in the $P_1$ case, we
see that the leading term in the expression for
$a_{m,F}$ is
\begin{equation}
\frac{7p^k(p-1)}{3x}. \label{nr3}
\end{equation}

\par As we did with $P_1$, we must now bound the number of fields $F$
with $\Eva(F/\Q)\cong P_3$ and $F/\Q$ ramified only at $p\infty$. As
above, let $L=F^{ab}$, the degree 3 subfield of $\Q(\zeta_p)$. The field $L$ then
 admits a $Q_8$ extension $F_0$, which has 3 intermediate fields
$L_i'$, each of which is a $C_2$ extension of $L$, and unramified over
$L$ by Lemma \ref{l1}. Let $K$ be the composite of the $L_i$. The field $L$ is
totally real, so by the following lemma, $K$ is as well.

\begin{lemma}\label{Q8}
Let $l$ be a totally real field and let $k/l$ be a Galois extension
with group $Q_8$. Then the unique intermediate field that is degree 4
over $l$ is totally real (the field is unique because $Q_8$ has a
unique subgroup of order 2). 
\end{lemma}
\begin{proof}
Let $k'=k \cap \R$. If $k'=k$, we are done. Otherwise, $[k:k']=2$, and
$k'$ is the unique intermediate field of degree 4 over $l$.
\end{proof}
Returning to the notation before Lemma \ref{Q8}, $K/L$ is unramified and $K$ is totally real, so the
number of choices for $K$ is at most $r_{2,2}(k)$, where $k$ is the
$2$-rank of $\Cl(L)$. Using \cite[Corollary 4]{realcubic}, we have
$|\Cl(L)|<\frac{1}{2}p$, so $k<\log_2{p}-1$. The field $F_0$ is a degree 2
extension of $K$, and the extension $F_0/K$ is at most tamely ramified, so $F_0
\subseteq K^{\mathfrak{p}\infty}$, where $\p$ is the prime of $L$ above
$p$, viewed as an ideal of
$K$. By class field theory, $\p$ splits into a product of 4 primes in $K$ since it
is principal in $L$. Using the exact sequence \eqref{exact} applied to
$\Cl^{\p\infty}(K)$, we find that $\Eva(K^{\p\infty}/K)$ has $2$-rank at
most 4 + 12 + the $2$-rank of $\Cl(K)$. The order of $\Cl(K)$ is at
most $\frac{22.2p^8}{\pi^{12}}$ \cite[page 95]{moonmonomial}. Set $\beta=\log_2(2^{12}\cdot\frac{22.2p^8}{\pi^{12}})$. If $F_0/K$ is
unramified above $p$, there are at most 
\begin{equation}
r_{2,1}(\beta)r_{2,2}(\log_2{p}-1) \label{beta1}
\end{equation}
choices for $F_0$, and
if $F_0/K$ is ramified above $p$, there are at most 
\begin{equation}
\left(r_{2,1}(4+\beta)-r_{2,1}(\beta) \right)r_{2,2}(\log_2{p}-1). \label{beta2}
\end{equation}

\par The final step to bound the number of possible $F_{\rho}$ is to
look at unramified $C_3 \times C_3$ extensions of $F_0$ (which give our field
$F_{\rho}$). This amounts to finding $r_{3,2}(l)$, where $l$ is the
$3$-rank of $\Cl(F_0)$. Using \cite[page 95]{moonmonomial} and the fact that
ramification in the infinite places does not contribute to the $3$-rank, we obtain 
\begin{equation}
l\leq \log_3{\frac{22.2p^{16}}{\pi^{24}}} \label{log1}
\end{equation}
if $F_0/K$ is unramified, and 
\begin{equation}
l \leq \log_3{\frac{22.2p^{20}}{\pi^{24}}} \label{log2}
\end{equation}
if $F_0/K$ is ramified. 

We have $x=1$ or $2$, depending on
whether $F_0/K$ is ramified. To obtain our bound for the number
$a_{m}$ (where $m=3k+3$) of $\rho$ ramified only at $p\infty$ with $\Eva(F_{\rho}/\Q)\cong P_3$ and $N(\rho)=p^m$, we need to
combine \eqref{nr3} with \eqref{beta1} and \eqref{log1}, and  combine
\eqref{nr3} with \eqref{beta2} and \eqref{log2}. Doing so, we find that
the leading term of $a_m$ in $p$, which comes only from the $x=2$
case, is bounded by
\[
\frac{605134.5p^{k+51}}{\pi^{50}}.
\]

Finally, to conclude the section on $P_3$, we find that for the number
$B_m$ of three-dimensional Artin representations with projective image
$P_3$ and conductor dividing $p^m$, we have

\begin{theorem}\label{t5}
Notation as above, $$B_m \lessapprox
\frac{605134.5p^{\frac{m}{3}+50}}{\pi^{50}}.$$
\end{theorem} 

\section{Comparison to the two-dimensional case}
For comparison to the two-dimensional case, we look at
\cite{siman}. Looking at Theorem 3 in \cite{siman} and multiplying by $\phi(p^m)=p^{m-1}(p-1)$ to account for all
possible $\chi$, we see that the bounds for the number of two-dimensional representations with conductor $p^m$ and projective image
isomorphic to $A_4$ or $S_4$ (and conjecturally for $A_5$) is on the order of 
$Cp^{2m-\epsilon}\log^4{p^m}$ for $\epsilon=\frac{1}{6},
\frac{1}{8}, \textrm{ or } \frac{1}{12}$ and $C$ a constant, both
$\epsilon$ and $C$ depending on the projective image of the
representation. Taking $p \to \infty$ for fixed $m$, we find that the
bounds in \cite{siman} are better than our bounds for the imprimitive
case. By going through the proofs of Theorems \ref{t4} and \ref{t5} in
this paper, we
may bound the number of $\rho$ with $\Imm \pi \circ
\rho = P_1$, $N(\rho)\mid p^m$, and with   $\Imm \pi \circ
\rho = P_3$, $N(\rho)|p^m$, respectively, as $m \to
\infty$ for fixed $p$. These bounds can be given by easily computable constants
multiplied by the bounds given in Theorems \ref{t4} and
\ref{t5}. Therefore, our bounds as $m \to \infty$ are stronger than
those in \cite{siman}, and our bounds are stronger
as either $p$ or $m$ tends to $ \infty$ for the primitive case. Of course, Wong's bounds
hold for representations of any conductor $N$, not only for
representations of prime power conductor.


\begin{acknowledgements}\label{ackref}
The Author would like to thank his advisor, Joe Silverman, for his
  helpful conversations and Siman Wong for his correspondence
  concerning this work in relation to his own. The author is also appreciative of
  the referee's comments and suggestions, which were detailed and thoughtful.
\end{acknowledgements}


%
%

%
\end{document}